\documentclass{article}
\usepackage{Paper_Style}
\usepackage{amsmath}
\usepackage{comment}
\usepackage{color}

\usepackage{amsmath,amssymb,amsthm,amsfonts,enumitem}
\usepackage{color}
\usepackage{graphicx,subcaption }
\usepackage{hyperref}
\usepackage{epstopdf}
\usepackage{bm}
\usepackage{mathtools}
\usepackage{tikz-cd}

\usepackage{tikz-cd}
\usepackage[utf8]{inputenc}

\theoremstyle{plain} 
\newtheorem{thm}{Theorem}

\newtheorem{prop}[thm]{Proposition}

\theoremstyle{definition}
\newtheorem{defn}{Definition}

\theoremstyle{remark}

\def\B{\mathbb{B}}

\newcommand{\Per}{\text{Per}}
\newcommand{\p}{\partial}
\newcommand{\n}{\nabla}
\newcommand{\CC}{\mathcal{C}}
\newcommand{\eps}{\varepsilon}

\newcommand{\red}[1]{\textcolor{red}{#1}}

\newcommand{\sauce}{\red{Source}}

\newcommand{\z}{\mathcal Z}
\newcommand{\vz}{\operatorname{VZ}}
\newcommand{\vc}{\operatorname{VC}}
\renewcommand{\B}{\mathcal B}
\newcommand{\M}{\mathbf M}

\begin{document}

\title{Mountain Pass Critical Points of the Volume Constrained Area Functional}
\date{\today}
\author{Gregory R. Chambers, Jared Marx-Kuo}
\maketitle

\begin{abstract}
We construct mountain pass critical points of the perimeter functional on sets of fixed volume. For a generic metric, this gives rise to a smooth almost embedded hypersurface with non-zero constant mean curvature. Our work utilizes recent techniques of Mazurwoski--Zhou \cite{mazurowski2024infinitely}, and a new result on the connectedness of Cacciopoli sets: any two smooth Cacciopoli sets can be connected by an $\mathbf{F}$-continuous map.
\end{abstract}

\section{Introduction}
Min-max is a recurring, powerful tool in differential geometry that has been used to find minimal surfaces, constant mean curvature surfaces, capillary surfaces, prescribed mean curvature surfaces, and the like. While we are unable to provide a holistic overview, we highlight the work of Almgren \cite{almgren1962homotopy, almgren1965theory}, Pitts \cite{pitts2014existence}, and Schoen--Simon \cite{schoen1981regularity} to find minimal surfaces, viewed as critical points of the area functional. The Almgren--Pitts program was revived by Marques--Neves in \cite{marques2014min, marques2017existence}, and lead to a series of works describing min-max hypersurfaces associated to the ``p-widths", $\{\omega_p\}$, as defined by Gromov \cite{gromov2006dimension}. We refer the reader to the introduction of Mazurowski--Zhou \cite{mazurowski2024infinitely} for more background, and we highlight the following celebrated result of min-max methods, known as ``Yau's conjecture":
\begin{theorem}[\cite{marques2017existence, irie2018density, chodosh2020minimal, song2023existence}]
Given $(M^{n+1}, g)$ a closed manifold and $3 \leq n+1 \leq 7$, there exists infinitely many distinct, smooth, minimal hypersurfaces.
\end{theorem}
Within min-max, one can focus on one-parameter mountain pass methods to construct minimal surfaces and their aforementioned variants (see e.g. \cite{bellettini2020inhomogeneous, mazurowski2024infinitely, marques2017existence, stern2021existence, zhou2018existence, dey2023existence} among many others). Viewing these distinguished surfaces as boundaries of sets, we develop a mountain pass theorem to show the existence of \textbf{volume constrained} critical points of the perimeter functional. \nl 
\indent One source of inspiration for our work lies in the recent results of Mazurowski--Zhou \cite{mazurowski2024infinitely} who construct critical points of the perimeter functional for sets of half volume.  Following the Marques--Neves program, Mazurowski--Zhou introduced the ``half-volume spectrum", $\{\tilde{\omega}_p\}$, which arises by considering sweepouts of the space of half-volume Caccioppoli sets. The $\Z_2$-invariance of this space allowed the authors to prove the following:
\begin{theorem}[Thm 1.1 \cite{mazurowski2024infinitely}] \label{mzInfinitelyManyThm}
Assume $M^{n+1}$ is a closed manifold of dimension $3\le n+1\le 5$. Let $g$ be a generic Riemannian metric on $M$. Then for each $p\in \N$ there exists an open set $\Omega_p\subset M$ with $\vol(\Omega_p) = \frac{1}{2}\vol(M)$ such that $\p \Omega_p$ is smooth and almost embedded, has non-zero constant mean curvature, and satisfies $\text{Area}(\p \Omega_p) = \tilde \omega_p(M)$.
\end{theorem}
The authors do this by considering a modified functional, $E_k$, which acts on the space of all Caccioppoli sets, $\CC(M)$, but penalizes volumes far from $\text{Vol}_g(M)/2$. This allows the authors to regularize and get a priori bounds on the mean curvature of the boundary of critical points. Increasing the penalty ($k \to \infty$) and taking appropriate limits allows them to recover a smooth, almost embedded CMC surface bounding half the volume of the manifold. \nl 
%
%
\indent We prove an adaptation of Theorem \ref{mzInfinitelyManyThm} to the one parameter setting in Theorem \ref{MPassThm}. Suppose $\Omega_1, \Omega_2$ are of the same volume, $V_0$, so that $\partial \Omega_i$ are constant mean curvature and strictly stable with respect to volume preserving deformations. By considering the collection of all paths from $\Omega_1 \to \Omega_2$, the perimeter necessarily increases due to strict stability, and hence there should exist a classical mountain pass by minimizing (over all paths) the maximum of the perimeter of all sets along each path. This would yield another set $\Omega$ of volume $V_0$ and constant mean curvature. \nl 
%
%
\indent Our result is analogous to the following mountain pass construction of  De Lellis--Ramic \cite[Cor 1.9]{de2018min}:
%
%
\begin{theorem}[\cite{de2018min}]
Let $(M^{n+1}, \p M, g)$ a smooth manifold with boundary satisfying mild assumptions and $\gamma^{n-1} \subseteq \p M$ a smooth hypersurface in the boundary. Suppose that 
\begin{enumerate}
    \item[(i)] There are two distinct smooth, oriented minimal embedded hypersurfaces, $\Sigma_0$ and $\Sigma_1$ with $\p \Sigma_0 = \p \Sigma_1 = \gamma$ which are strictly stable, meet only at the boundary and bound some open domain $A$
\end{enumerate}
Then there exists a third distinct embedded minimal hypersurface, $\Gamma$, with $\p \Gamma = \gamma$.
\end{theorem}
\noindent In fact, De Lellis--Ramic concieve of $\Sigma_0, \Sigma_1$ as local minimizers of the plateau problem, in the same way that we can apply Theorem \ref{MPassThm} when  $\Omega_1, \Omega_2$ are local minimizers of the isoperimetric inequality for a fixed volume $V_0$. The intuition of their theorem is the same as \ref{MPassThm}, however the authors go to great length to establish regularity at the boundary. The analogous regularity theory for Theorem \ref{MPassThm} was mostly developed by Mazurowski--Zhou in \cite{mazurowski2024infinitely}. \nl 
\indent We note that our mountain pass theorem \textit{assumes} that there exists a non-trivial path from $\Omega_1 \to \Omega_2$. This raises the following questions
\begin{enumerate}
\item Is the space of all Cacciopoli sets path connected? With respect to what norms? 
\item Is the space of Cacciopoli sets with \textit{fixed volume} path connected?
\end{enumerate}
To the authors knowledge this topological investigation of $\CC(M)$ has not been discussed prior, and we give partial answers to the above in the setting of Cacciopoli sets with smooth boundary in Theorem \ref{pathConnectedThm}. We hope to address the non-smooth case in future work.

\subsection{Statement of Main Results}

\subsubsection{Connectedness of $\CC^{\infty}(M)$}
Let $\CC^{\infty}(M)$ denote the space of Cacciopoli sets with smooth boundary and $\CC_{V_0}^{\infty}(M)$ the space of such sets with a fixed volume, $V_0$.
\begin{theorem} \label{pathConnectedThm}
Let $(M^{n+1}, g)$ be a closed manifold (of any dimension $n+1 \geq 2$). Then for any $A, B \in \CC^{\infty}(M)$ (respectively $\CC^{\infty}_{V_0}(M)$), there exists a path 
\begin{align*}
\sigma&: [0,1] \to (\CC(M), \; \mathbf{F}) \\
\sigma(0) &= A, \quad \sigma(1) = B
\end{align*}
when $A, B \in \CC^{\infty}_{V_0}(M)$, we further have that $|\sigma(t)| = V_0$ for all $t \in [0,1]$.
\end{theorem}
\indent We emphasize the usage of the $\mathbf{F}$ (``Bold-F") norm in Theorem \ref{pathConnectedThm}, which is typically used in min-max constructions (see e.g.\cite{li2019introduction, marques2020applications}) in order to keep track of the varifold induced by the boundary, $|\partial \Omega|$, as well as the set itself via continuity of $\chi_{\Omega}$ in $L^1(M)$, also known as the $\mathcal{F}$ (``Flat") norm. See \S \ref{GMTBackground} for the formal definitions.  \nl 
\indent We note that if we restrict to the weaker $\mathcal{F}$-norm, then Theorem \ref{pathConnectedThm} holds easily. With no volume constraint, one can simply cover any $A \in \CC(M)$ with a finite union of balls and continuously removing each ball to a point, creating a path between $A \to \emptyset$. In the fixed volume setting, one can also create a path between $A, B \in \CC_{V_0}^{\infty}(M)$ by covering the two sets with small balls and removing balls from $A \backslash B$ while including balls from $B \backslash A$. However, removing small balls may not be continuous in the $\mathbf{F}$-norm if the balls are chosen poorly. As an example, consider excising $B_1(0)$ via the family of annuli $\sigma(t) = B_1(0) \backslash B_{1-t}(0)$. As $t \to 1$, $\mathcal{F}(\sigma(t), \emptyset) \to 0$ but $\mathbf{F}(\sigma(t), \emptyset) \to 2$ (see Figure \ref{fig:varifoldex}).
\begin{figure}[h!]
\centering
\includegraphics[scale=0.5]{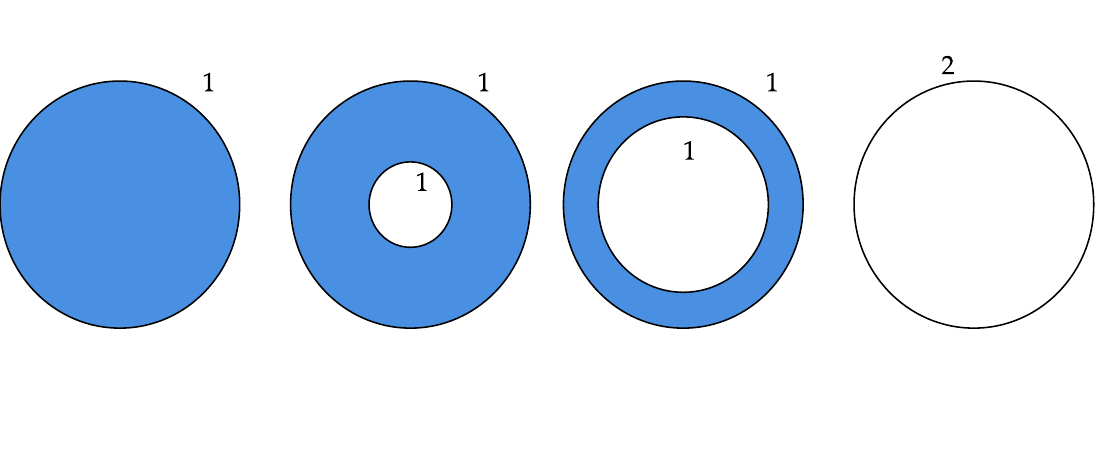}
\caption{Continuity in $\mathcal{F}$ does not imply continuity in $\mathbf{F}$.}
\label{fig:varifoldex}
\end{figure}
We also note that Theorem \ref{pathConnectedThm} can also be generalized to ambient manifolds which are non-compact, assuming we restrict to smooth, precompact Cacciopoli sets. \nl 
\indent Finally, we remark that the above statement is different than showing that the space of all smooth Cacciopoli sets is path connected with respect to the $\mathbf{F}$ norm - indeed this is manifestly false if the topology of the sets is different. We emphasize that for any $A, B \in \CC^{\infty}(M)$ (or $\CC^{\infty}_{V_0}(M)$) we are able to produce $\sigma: [0,1] \to \CC(M, \mathbf{F})$ such that $\sigma(0) = A$ and $\sigma(1) = B$. \newline 
\indent The result for $\CC^{\infty}(M)$ follows easily from choosing a Morse function on $\Omega \in \CC^{\infty}(M)$ with $f:\Omega \to \R$, $f^{-1}(1) = \partial \Omega$ and $f^{-1}(0) = p$ a point. As a quick corollary of Theorem \ref{pathConnectedThm}, we have that any $\Omega \in \CC^{\infty}(M)$ can be contracted to the empty set in an $\mathbf{F}$- continuous manner. 
However, the result for $\CC^{\infty}_{V_0}(M)$ is more delicate, as if we consider $A, B \in \CC_{V_0}^{\infty}(M)$ with $A \cap B \neq \emptyset$, then it is unclear how to use a Morse function to deform $A \to B$ while preserving total volume. 
%
%
%
%
%
\subsubsection{Perimeter Mountain Pass}
Inspired by \cite{de2018min} and \cite{mazurowski2024infinitely}, we prove the following:
\begin{theorem}\label{MPassThm}
Let $(M^{n+1}, g)$ be a closed manifold with $3 \leq n + 1 \leq 5$ and generic metric. Suppose that $\Omega_1, \Omega_2 \in \CC(M)$ are distinct sets with the same volume, $h_0$, with smooth boundary. Further assume that $\mathbf{M}(\p \Omega_1) \geq \mathbf{M}(\p \Omega_2)$ and $\p \Omega_1$ is a strictly stable critical point for the area functional under smooth volume preserving deformations. \nl 
\indent There exists $\Omega \in \CC(M)$ with $\text{Vol}(\Omega) = h_0$ and $\p \Omega$, a critical point of the (volume constrained) area functional, with $A(\p \Omega) > A(\p \Omega_1)$. $\p \Omega$ has constant mean curvature, $c \neq 0$, and is smooth and almost-embedded. 
\end{theorem}
\noindent We sketch the theorem in a few steps
\begin{enumerate}
\item We show the existence of a non-empty homotopy class of paths, $\gamma: [0,1] \to \CC(M)$ with $\gamma(0)= \Omega_1$, $\gamma(1) = \Omega_2$, and $|\gamma(t)| = h_0$ for all $t$. This uses Theorem \ref{pathConnectedThm}.
\item We define an $L$-value for said homotopy class.
\item We use strict stability and quantitative bounds on perimeter (cf. \cite{chodosh2022riemannian, inauen2018quantitative}) to show that $L > \max(P(\Omega_1), P(\Omega_2))$. 
\item Adapting the work of Mazurowski--Zhou \cite{mazurowski2024infinitely}, we define an $L^k$-value for the homotopy class of all paths, $\gamma: [0,1] \to \CC(M)$ with $\gamma(0)= \Omega_1$, $\gamma(1) = \Omega_2$, where $L^k$ corresponds to the perimeter functional plus a volume penalty.
\item Sending $k \to \infty$, we show that $L^k \to L$, and use the regularity theory of Mazurowski--Zhou to show the existence of a fixed volume, CMC hypersurface which achieves the $L$-value. 
\end{enumerate}
It would be interesting to study the index of $\partial \Omega$ with respect to volume preserving deformations. Intuitively, the index should be $1$, following the program of Marques--Neves (see e.g. \cite{marques2021morse}). However, the difficulty of doing min-max over a space of Cacciopoli sets with fixed volume prevents a direct index upper bound.

\section{Preliminaries}
\subsection{Geometric Measure Theory} \label{GMTBackground}
In this section, we recall concepts from geometric measure theory needed in the paper, following the conventions from Mazurowski--Zhou \cite{mazurowski2024infinitely}. See \cite{simon1984lectures} for more detail. Let $(M^{n+1},g)$ be a closed Riemannian manifold. 
\begin{itemize}
\item Let $\CC(M)$ denote the space of all Caccioppoli sets in $M$. For $\Omega \in \CC(M)$, let $|\Omega| = \text{Vol}(\Omega)$.
\item Let $\CC^{\infty}(M)$ denote the space of all Cacciopoli sets with smooth boundary.
\item For $h_0 \in (0, \text{Vol}(M))$, let $\CC_{h_0}(M)$ denote the space of all $\Omega \in \CC(M)$ with $\text{Vol}(\Omega) = h_0$ and $\CC_{h_0}^{\infty}(M)$ the space of such sets with smooth boundary.

\item Let $\mathcal V(M)$ denote the space of all $n$-dimensional varifolds on $M$. 
\item Let $\z(M,\Z_2)$ denote the space of $n$-dimensional flat cycles in $M$ mod 2. 
\item Given $\Omega \in \CC(M)$, the notation $\p \Omega$ denotes the element of $\z(M,\Z_2)$ induced by the boundary of $\Omega$. 
\item For $\Omega \in \CC(M)$, $\text{Per}(\Omega)$ denotes the perimeter.
\item Given $T \in \z(M,\Z_2)$, the notation $\vert T\vert$ stands for the varifold induced by $T$. 
\item Let $\B(M,\Z_2)$ denote the space of all $T\in \z(M,\Z_2)$ such that $T = \p \Omega$ for some $\Omega \in \CC(M)$. 
\item Let $\z_{h_0}(M,\Z_2)$ denote the space of all $T\in \z(M,\Z_2)$ such that $T = \p \Omega$ for some $\Omega \in \CC_{h_0}(M)$.
\item We use $\mathcal F$ to denote the flat topology, $\mathbf F$ to denote the $\mathbf F$-topology, and $\mathbf M$ to denote the mass topology.  By convention, the $\mathcal F$ norm on $\CC(M)$ is given by 
\[
\mathcal F(\Omega_1, \Omega_2) = |\Omega_1 \Delta \Omega_2| = ||\chi_{\Omega_1} - \chi_{\Omega_2}||_{L^1(M)}
\]
and the $\mathbf F$ norm on $\CC(M)$ is given by
\[
\mathbf F(\Omega_1,\Omega_2) = \mathcal F(\Omega_1,\Omega_2) + \mathbf F(\vert \p \Omega_1\vert,\vert \p \Omega_2\vert).
\]
As we will be working explicitly with the $\mathbf{F}$ norm, we recall its definition from Pitts \cite[P. 66]{pitts2014existence} for varifolds $V,W$
\[
\mathbf{F}(V, W) = \text{sup} \{V(f) - W(f) \; : \; f \in C_c(G_k(M)), \;\; |f| \leq 1, \; \text{Lip}(f) \leq 1 \}
\]

\item Let $\vz(M,\Z_2)$ denote Almgren's $\vz$ space (see \cite{almgren1965theory}).
\item Let $\vc(M)$ denote Almgren's $\vc$ space (see \cite{almgren1965theory} and \cite{wang2023existence}).  
\end{itemize} 

The set $\vz(M,\Z_2)$ \cite{almgren1965theory} consists of all pairs $(V,T) \in \mathcal V(M) \times \mathcal \B(M,\Z_2)$ such that there is a sequence $T_k \in \mathcal \B(M,\Z_2)$ with $\vert T_k\vert \to V \in  \V(M)$ and $T_k\to T \in \mathcal B(M,\Z_2)$.  Note that it may or may not be true that $V = \vert T \vert$, but it is always true that $\|\, \vert T\vert \, \|\leq \|V\|$ as measures. An example of strict inequality occurs if we let $\gamma$ denote an equator of $S^2$ with the round metric, and note that $(2 |\gamma|, 0) \in \vz(M, \Z_2)$, as it is a limit of $T_k = S^2 \backslash N_{1/k}(\gamma)$ where $N_{1/k}(\gamma)$ is a tubular neighborhood of $\gamma$ of distance $1/k$. \nl 
%
%
\indent We endow $\vz(M, \Z_2)$ with the product metric, so that for any $(V, T), (V', T')\in \vz(M, \Z_2)$, the $\mathscr{F}$-distance between them is
\[ \mathscr{F}\big( (V, T), (V', T') \big) = \mathbf{F}(V, V') + \mathcal{F}(T, T'). \]
We will use $\vz(M, \mathscr{F}, \Z_2)$ if we wish to emphasize the metric $\mathscr{F}$.

The VC space is entirely analogous but with $\mathcal B(M,\Z_2)$ replaced by $\CC(M)$. The set $\vc(M)$ consists of all pairs $(V,\Omega) \in \mathcal V(M) \times \CC(M)$ such that there is a sequence $\Omega_k \in \CC(M)$ with $\vert \p \Omega_k\vert \to V \in  \V(M)$ and $\Omega_k\to \Omega \in \mathcal C(M)$.    We similarly endow $\vc(M)$ with the product metric, so that for any $(V, \Omega), (V', \Omega')\in \vc(M)$, the $\mathscr{F}$-distance between them is
\[ \mathscr{F}\big( (V, \Omega), (V', \Omega') \big) = \mathbf{F}(V, V') + \mathcal{F}(\Omega, \Omega'). \]
\noindent We will write $\vc(M, \mathscr{F})$ if we wish to emphasize the metric $\mathscr{F}$. \newline 
\indent As noted in the introduction of Mazurowski--Zhou \cite{mazurowski2024infinitely}, the $\vz$ and $\vc$ spaces are convenient for considering min-max with a volume constraint, as a ``pull--tight" procedure for varifolds arising as boundaries of Cacciopoli sets with constrained volume appears difficult to produce. Moreover, these spaces satisfy similar compactness properties (see e.g. \cite[Prop 2.1, 2.2]{mazurowski2024infinitely}). 

\subsection{Acknowledgements}
The authors thank Otis Chodosh for useful discussions in this problem. The second author is supported by an NSF grant, 23-603.

\section{Connectedness of Cacciopoli Sets}
In this section we prove Theorem \ref{pathConnectedThm}. The result for $\CC^{\infty}(M)$ follows by taking a Morse function on the set of interest and contracting via level sets, so we will focus on the result for $\CC_{V_0}^{\infty}(M)$. To show that $\CC_{V_0}^{\infty}(M)$ is path connected, we consider $A, B \in \CC_{V_0}^{\infty}(M)$ and perturb $A$ and $B$ so that their boundaries are transverse.  We will ``shrink" $A \setminus B$ and ``grow"
$B \setminus A$ while preserving $A \cap B$, so that the total volume does not change along this path. We first grow $A \cap B$ to $A$. To do this, we
\begin{enumerate}
\item Grow $\partial A \cap \partial B$ to $\partial A$ using the fact that the boundary is smooth, and by using Morse-theoretic techniques. 
\item Extend this procedure to a smaller inner tubular neighborhood of $\partial A \setminus B$.
\item Extend this to the entirety of $A$.
\end{enumerate}
Call this path $\sigma_1$. We do the same to go from $B \cap A$ to $B$, forming a path, $\sigma_2$. Combining these paths so that the total volume does not change, we obtain the desired path. 
\nl 
\indent We first show that given two smooth sets $A,B \in \CC_{V_0}^{\infty}(M)$, they can be deformed to have transverse intersection.
\begin{lemma}[Generic Transversality] \label{genericTransverse}
Suppose $A, B \in \CC^{\infty}_{V_0}(M)$, then $A, B$ can be deformed in an $\mathbf{F}$-continuous manner to $\tilde{A}, \tilde{B} \in \CC^{\infty}_{V_0}(M)$ such that $\partial \tilde{A}$ is transverse to $\p \tilde{B}$.
\end{lemma}
\begin{proof}
First suppose that $\partial A \neq \partial B$. As these are smooth hypersurfaces, we know that $\partial A \backslash \overline{B}$ contains an open ball $U \subseteq \partial A$ with $\mathcal{H}^{n-1}(U) \neq 0$. Let 
\begin{align*}
F_A&: \partial A \times [-\delta, \delta] \to M \\
F_A(s,t) &= \exp_{s}(t \nu(s))
\end{align*}
be a Fermi coordinate parameterization of a tubular neighborhood of $\partial A$. By generic transversality results, there exists $f: \partial A \to \R$ such that $||f||_{C^1}$ can be made arbitrarily small and define $\tilde{A}$ so that 
\[
\partial \tilde{A} = \{F(s, f(s)) \; | \; s \in \partial A \}
\]
and $\partial \tilde{A}$ is transverse to $\partial B$. \nl 
\indent By potentially choosing a smaller ball $U' \subseteq U \subseteq \partial A$, we can actually further modify $f \Big|_{U'}$ without affecting transversality because $\partial B \cap U = \emptyset$. Now consider 
\begin{align*}
f&: [0,1] \times \partial A \to \R \\
f(t,s) &= (1 - \eta_{U''}(s)) t f(s) + \eta_{U''}(s) t p(t)
\end{align*}
Here, $\eta_{U''}(s)$ is a bump function which is $1$ on some smaller ball $U'' \subseteq U'$ and vanishes outside of $U'$. We then choose $p: [0,1] \to \R$ so that 
\[
A(t) = \text{Int} \left( \{F(s, f(t,s)) \; | \; s \in \partial A \} \right)
\]
has constant volume equal to $V_0$. Note that such a $p(t)$ exists by choosing our generic perturbation, $f$, to have arbitrarily small $C^1$ norm. Note that $A(t)$ is an $\mathbf{F}$-continuous deformation of sets because it is a graphical perturbation of the boundary. Thus $\tilde{A} = A(1)$ is the desired set with $A(t)$ being the desired $\mathbf{F}$-continuous deformation. \nl 
\indent If $\partial A = \partial B$, then choose two disjoint open balls $U_1, U_2 \subseteq \partial A$. We push $\p A$ in along $U_1$ and out along $U_2$ so that the resulting $\tilde{A}$ satisfies $\p \tilde{A} \neq \p B$. We can do this while preserving the total volume. Consider 
\begin{align*}
f&: [0,1] \times \partial A \to \R \\
f(t,s) &= \eta_{U_1}(s) t -\eta_{U_2}(s) p(t)
\end{align*}
where $p: [0,1] \to \R^+$ is chosen so that 
\[
A(t) = \text{Int} \left( \{F(s, f(t,s)) \; | \; s \in \partial A \} \right)
\]
has constant volume for all $t \in [0,1]$. Letting $\tilde{A} = A(1)$, we can apply the previous argument of the lemma.
\end{proof}
\noindent We now establish the main lemma: any smooth Cacciopoli set can be contracted to the empty set in an $\mathbf{F}$-continuous way.
\begin{lemma} \label{contractLemma}
Suppose $\Omega^n \subseteq M^n$ a Cacciopoli set with smooth boundary and $n \geq 2$. Then there exists 
\[
\sigma: [0,1] \to \Omega
\]
Such that $\sigma(1) = \Omega$, $\sigma(0) = \emptyset$ and $\sigma$ is $\mathbf{F}$-continuous. When $n = 1$, such a map exists but it is only $\mathcal{F}$-continuous.
\end{lemma}
\begin{proof}
Choose a Morse function $f: \Omega \to \R$ such that $f^{-1}([0,1]) = \Omega$ and $f^{-1}(0) = p$ where $p$ is a fixed interior point. Such a Morse function is a consequence of work of Milnor (see \cite[Thm 2.5]{milnor2025lectures}). We note that for $s \neq 0$, $\partial f^{-1}([0,t]) \rightarrow \partial f^{-1}([0,s])$ for $t \to s$ and the convergence occurs smoothly away from finitely many points. When $s = 0$, $\p f^{-1}([0,t])$ converges in a manner of a sphere converging to a point, which is also $\mathbf{F}$-continuous. Note that because $n \geq 2$ and $\p f^{-1}([0,t])$ is $n-1$-dimensional, points are at least codimension $1$ and smooth convergence away from a set of codimension $\geq 1$ implies convergence in $\mathbf{F}$. Now let $\sigma(t) = f^{-1}([0,t])$. \nl 
\indent When $n = 1$, any Cacciopoli set is necessarily diffeomorphic to a finite union of intervals whose closure is disjoint. In this case, there exists a path $\sigma: [0,1] \to \CC(M)$ which contracts each of these intervals to a point, though this is only $\mathcal{F}$-continuous and not $\mathbf{F}$-continuous.
\end{proof}
\noindent We now show that given $A, B \in \CC_{V_0}^{\infty}(M)$, we can separate $A \backslash B$ from $B$ in an $\mathbf{F}$-continuous way.
\begin{lemma}[Carving Lemma] \label{lem:carving}
Let $n \geq 2$. Let $A,B \in \CC_{V_0}^{\infty}(M)$ such that $\partial A \cap \partial B$ is a $(n-2)$-dimensional smooth manifold. Then there exists a continuous path
\[
\sigma: [0,1] \to (\CC(M), \mathbf{F})
\] 
such that $\sigma(0) = A$, $\sigma(t) \subseteq A$ for all $t$, and $\sigma(1) = S_A \sqcup B$, where $\overline{S_A} \cap \overline{B} = \emptyset$ and $S_A \subseteq A \backslash B$ is itself a smooth Cacciopoli set. Moreover, $\sup_{t \in [0,1]}\mathbf{F}(\sigma(t), A)$ can be taken to be arbitrarily small.
\end{lemma}
\begin{proof}[Proof of Lemma~\ref*{lem:carving}]
First suppose $n \geq 3$. Let $S = A \cap  \p B$; $S$ is an $n-1$-dimensional manifold with smooth boundary given by $\p S = \p A  \cap \p B$. Let $\delta > 0$ and consider $S_{\delta} = S \backslash N_{\delta}(\partial S)$. Let 
\begin{align*}
N&: S_{\delta} \times [-\eta, \eta] \to A \\
N(s, t) &= \exp_{s}(t \nu_S(s))
\end{align*}
i.e. the normal exponential map in Fermi coordinates on $S \subseteq A$, where $\nu_S$ points towards $A \backslash B$. Note that for each $\delta > 0$, there exists an $\eta > 0$ such that the above is well defined. By convention, we choose $\eta < \delta$. Now we can ``carve out" $S_{\delta}$ (i.e. produce an $\mathbf{F}$-continuous map from $\emptyset \to S_{\delta}$) using $\sigma_{\delta}:[0,1] \to S_{\delta}$ as in Lemma \ref{contractLemma}, noting that for $\delta$ sufficiently small $\partial S_{\delta}$ will be smooth. Now consider the map
\begin{align*}
F_{\delta}&: [0,1] \to A \backslash B \\
F_{\delta}(t) &= \sigma_{\delta}(t) \times [0, \eta] 
\end{align*}
where in the right hand side we abuse notation writing Fermi coordinates in a product form. Clearly, $F_{\delta}$ is continuous in the $\mathcal{F}$ norm. To see continuity in the $\mathbf{F}$ norm, note that 
\begin{align*}
\partial F_{\delta}(t) &= \partial \sigma_{\delta}(t) \times [0, \eta] \sqcup \sigma_{\delta}(t) \times \{-\eta\} \sqcup \sigma_{\delta}(t) \times \{\eta\} \\
||\partial F_{\delta}(t)|| &= ||\partial \sigma_{\delta}(t) \times [0, \eta]|| + ||\sigma_{\delta}(t) \times \{0\}|| + ||\sigma_{\delta}(t) \times \{\eta\}||
\end{align*}
Because $\sigma_{\delta}(t)$ is $\mathbf{F}$-continuous, each varifold summand on the right hand side varies continuously with $t$. \nl 
\indent When $n = 2$, the map $F_{\delta}(t)$ is not $\mathbf{F}$-continuous, however we can still produce an $\mathbf{F}$-continuous map 
\begin{align*}
G_{\delta}&: [0,1] \to A \\
G_{\delta}(0) &= \emptyset \\
G_{\delta}(1) &= S_{\delta} \times [0, \eta]
\end{align*}
Note that since $n = 2$, $S$ and $S_{\delta}$ are both diffeomorphic to a finite union of intervals. So in order to construct an $\mathbf{F}$-continuous map which begins at the emptyset and ends at $S_{\delta} \times [0,\eta]$, it suffices to construct an $\mathbf{F}$-continuous map from the empty set to $[0,1]^2$ (since being $\mathbf{F}$-continuous is diffeomorphism invariant). One such example is $f(t) = [0,1]^2 \cap \{x^2 + y^2 \leq 2t\}$. \nl 
\indent For the remainder of the proof, we now assume $n \geq 2$. We proceed to carve out a non-uniform neighborhood of $S \backslash S_{\delta}$. 
\begin{align*}
\gamma_{\delta}&: [0,1] \to A \\
\gamma_{\delta}(t) &= \{N\left(s, r \right) \; | \; s \in S \backslash S_{\delta}, \quad 0 \leq |r| \leq \overline{\eta}\left( \frac{\text{dist}(s, \partial S)}{\delta t} \right) \cdot \min \left(\eta, \text{dist}(s, \partial S)\right) \} 
\end{align*}
here, $\overline{\eta}: \mathbb{R}^{\geq 0} \to \mathbb{R}^{\geq 0}$ is a bump function which is $1$ on $[0,1]$, has bounded first derivative, vanishes for all $|x| \geq 2$, and is positive on $(-2, -1) \cup (1, 2)$. The idea is that for a fixed value of $t \neq 0$, a small tubular neighborhood is carved out for points close to $\partial S$. For points very far away, i.e. $\text{dist}(s, \partial S) > t \delta$, the size of the neighborhood gets smaller until it vanishes. By varying $t$ from $0$ (for which this map  just produces $S$) to $1$, we continuously carve out a non-uniform normal (non-tubular) neighborhood of $S \backslash S_{\delta}$ where the radius of the neighborhood gets smaller as one approaches $\partial S$. \nl
\indent $\gamma(t)$ is clearly continuous with respect to the $\mathcal{F}$ norm. To see continuity with respect to the $\mathbf{F}$-norm, note that
\begin{align*}
\partial \gamma_{\delta}(t) &= \{N(s, \eta(\text{dist}(s, \partial S)/\delta t) \cdot \text{dist}(s, \partial S)) \; | \; s \in S \backslash S_{\delta}\} \\
& \bigsqcup \{N(s,0) \; | \; s \in S \backslash S_{\delta}\} \\
& \qquad \bigsqcup \{N(s, r) \; | \; \text{dist}(s, \partial S) = 2 \delta t \; | \;  0 \leq |r| \leq \eta \left( \frac{\text{dist}(s, \partial S)}{\delta t} \right) \cdot \text{dist}(s, \partial S)\} \\
&= V_+(t) + V_-(t) + V_b(t)
\end{align*}
Given any $t_0\neq 0$, and any $s$ with $\text{dist}(s, \partial S) \neq 2 \delta t$, the convergence of $\text{supp}(V_{\pm})(t) \to \text{supp}(V_{\pm})(t_0)$ is graphical over $s$ hence $\mathbf{F}$-continuous locally, since $V_{\pm}$ has density $1$ on its support. Note that the graphical convergence fails for all 
\[
\text{dist}(s, \partial S) \in \{0, 2 \delta t\}.
\]
However, the set of the above such $s$ necessarily has $(n-2)$-dimensional Hausdorff measure - this is because $\delta$ is chosen sufficiently small so that the normal exponential map of $\partial S \to S$ is injective and smooth. Thus for the purpose of establishing continuity of the $(n-1)$-varifolds $V_{\pm}(t)$, this is irrelevant. A similar argument works for $V_b(t)$ - the set of 
\[
\{s \in S \backslash S_{\delta } \; | \; \text{dist}(s, \partial S) = 2 \delta t\}
\]
varies continuously in $t$ and so does the parameter of the normal coordinate
\[
0 \leq |r| \leq \eta \left( \frac{\text{dist}(s, \partial S)}{\delta t} \right) \cdot \text{dist}(s, \partial S)
\]
which establishes continuity of $V_b(t)$ with respect to $t$.  \nl 
\indent To see continuity at $t = 0$, note that for values close to $t = 0$, $\eta(\text{dist}(s, \partial S) / t \delta) \neq 0$ if and only if $\text{dist}(s, \partial S) < 2 \delta t$ is close to $0$. However in this case, it is clear that each of $V_{\pm}(t), V_b(t)$ converge to $0$ as varifolds as $t \to 0$. This follows as the normal coordinate is less than a constant times $\text{dist}(s, \partial S)$, and hence we can bound the varifold mass by a constant times $\text{dist}(s, \p S) < 2 t \delta$, which tends to $0$ with $t$. \nl 
\indent Now consider the conglomerate map
\[
\sigma(t) = A \backslash \begin{cases}
    \gamma_{\delta}(2t) & t \in [0,1/2] \\
    \gamma_{\delta}(1) \cup \left[ S_{\delta} \times [0, \eta] \backslash F_{\delta}(2 - 2t)\right] & t \in [1/2, 1]
\end{cases}
\]
which will be $\mathbf{F}$- continuous for all $t \in [0,1]$. See Figures \ref{fig:ngbdexcise} \ref{fig:exciseresult} for images of the resulting set.
\begin{figure}[h!]
\centering
\includegraphics[scale=0.7]{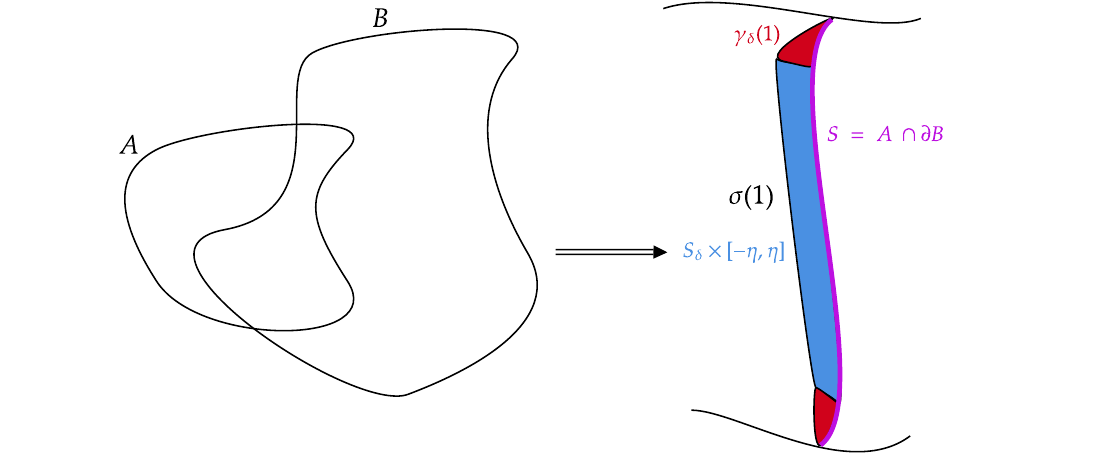}
\caption{Carving around $A \cap \partial B$ via $\sigma(t)$}
\label{fig:ngbdexcise}
\end{figure}
\begin{figure}[h!]
\centering
\includegraphics[scale=0.7]{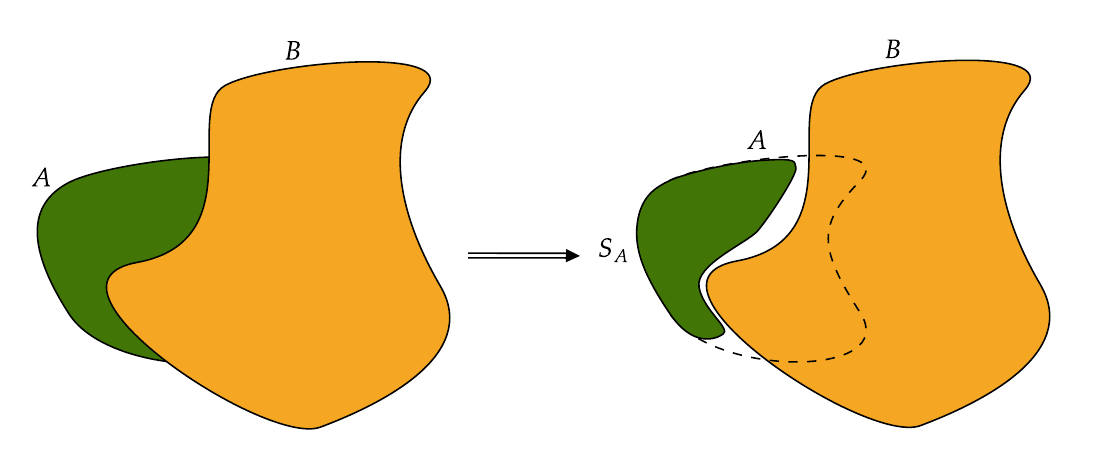}
\caption{$S_A$, the resulting set, visualized after smoothing. Note that $\sigma(1) = S_A \sqcup B$.}
\label{fig:exciseresult}
\end{figure}
Let $S_A = (A \backslash B) \cap \sigma(1)$. Note that $\overline{S_A} \cap \overline{A \cap B} = \p S$ and $S_A$ is a manifold with corners. We now smooth out these corners and separate $S_A$ from $\p S$ so that the smoothened set has a closure disjoint from $\p S$. Note that \cite[3.2]{CL} includes a section about parametrically smoothing manifolds with corners. We also describe a mechanism for this. \nl 
\indent Since $\p S = \partial A \cap \partial B$ is itself a smooth codimension $2$ manifold, we know that a neighborhood of $\p S$, $U(\p S)$, is diffeomorphic to $\p S \times D^2$, where coordinates on $D^2$ are given by $y(p) = \text{dist}(p, \partial A)$ and $z(p) = \text{dist}(p ,\partial B)$ are signed distances, taking the positive direction corresponding to the outer unit normals for both $\partial A, \partial B$. Moreover, we note that in these coordinates
\[
U(\p S) \cap (A \backslash B) = \{(q, y, z) \in \p S \times D^2 \; | \; y \leq 0, z \geq 0\}
\]
Thus, in order to smooth out $S_A \subseteq (A \backslash B)$, it suffices to find a smoothing of $U(\p S) \cap (A \backslash B)$ which lies inside of $A \backslash B$. From the above, using the global coordinates for $U(\p S) \cong \p S \times D^2$, it suffices to find a smoothing of the set 
\[
\{(y,z) \in \R^2 \; | \; y^2 + z^2 \leq 1, \; y \leq 0, z \geq 0\}
\]
near $y = 0 = z$. One way to do this would be via a smooth parameterized curve $(y(t), z(t))$ such that 
\begin{align*}
y, z&: (-\infty, \infty) \to \R \\
y(t) &= \begin{cases}
0 & \forall t < 0 \\
t - 1/2 & \forall t \geq 1
\end{cases} \\
z(t) &= \begin{cases}
-t + 1/2 & t \leq 0 \\
0 & t \geq 1
\end{cases} \\
y(0)&= 0, \; y(1) = 1/2, \; z(0) = 1/2, \; z(1) = 0 
\end{align*}
and $(y(t), z(t))$ is a smooth curve on $[0,1]$ which interpolates between the boundary conditions and so that the curve is convex. Let the entire curve be given by $\gamma(t) = (y(t), z(t))$ (see figure \ref{fig:desingcurve})
\begin{figure}[h!]
\centering
\includegraphics[scale=0.5]{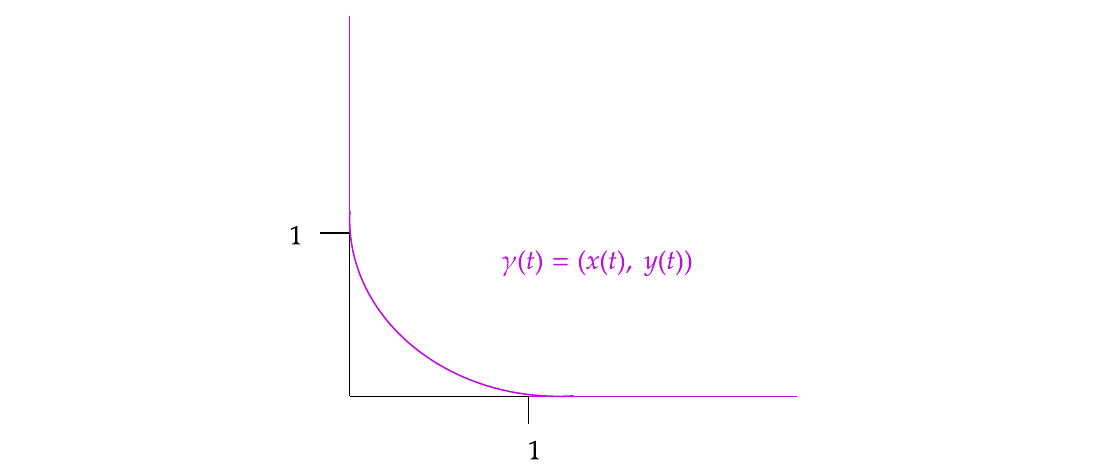}
\caption{Desingularization of manifold with corner}
\label{fig:desingcurve}
\end{figure}
and define $\eta \gamma(t) = (\eta y(t), \eta z(t))$ for any $\eta > 0$. Let $\nu_t$ be the normal to this curve which has positive $\p_y$ direction, and use this to define a signed distance to $\gamma(t)$. We now define 
\[
\Omega_{\eta} = \{(y,z) \in D^2 \; | \; \text{dist}((y,z), \eta\gamma) \geq 0 \}
\]
Clearly this is a set with smooth boundary given by $\eta \gamma$. Moreover, the boundaries converge graphically everywhere for all $\eta > 0$ and graphically everywhere except from $y = z = 0$ at $\eta = 0$. Thus $\{\Omega_{\eta}\}$ provide a one parameter sequence of sets whose boundaries converge continuously in the mass topology on varifolds. Thus the lemma is proved by using this desingularization of $U(\p S) \cap (A \backslash B)$ and then leaving the rest of the set unchanged. \nl 
\indent Note that if we choose $\delta$ arbitrarily small in our initial construction, then $\sup_{t \in [0,1]}\mathbf{F}(\sigma(t), A)$ can be made arbitrarily small as well.
\end{proof}
\subsection{Proof of Thm \ref{pathConnectedThm}}
Given $A, B \in \CC^{\infty}(M)$ it suffices to show that one can find an $\mathbf{F}$-continuous map from $A \to \emptyset$ and $B \to \emptyset$. This follows immediately from Lemma \ref{contractLemma}. \nl 
\indent Given $A, B \in \CC_{V_0}^{\infty}(M)$, WLOG suppose that $A \neq B$ else the constant path works. Apply Lemma \ref{genericTransverse} so that WLOG, we can assume $A, B$ have transversally intersecting boundaries. Choose $\eps < \min(V_0, |B \backslash A|, |A \backslash B|)/4$. Apply Lemma \ref{lem:carving} to the set $A$ and $B$ to find a map $\sigma_A: [0,1] \to \CC(M, \mathbf{F})$ so that $\sup_{t \in [0,1]} \mathbf{F}(\sigma_A(t), A) \leq \eps$. Moreover, let $\gamma_A: [0,1] \to \CC(M, \mathbf{F})$ be the map from Lemma \ref{contractLemma} with $\gamma_A(0) = S_A$ and $\gamma_A(1) = \emptyset$. Similarly apply Lemma \ref{lem:carving} to the sets $B$ and create a map $\sigma_B: [0,1] \to \CC(M, \mathbf{F})$ so that $\sup_{t \in [0,1]} \mathbf{F}(\sigma_B(t), B) \leq \eps$ and Lemma \ref{contractLemma} to get a map $\gamma_B: [0,1] \to \CC(M, \mathbf{F})$. We first ``separate" $A \backslash B$ from $A \cap B$ by using $S_A$. Consider the map
\begin{align*}
\sigma_1&: [0,1] \to \CC(M, \mathbf{F}) \\
\sigma_1(t) &= \sigma_A(t) \sqcup \gamma_B(f(t))
\end{align*}
where $f:[0,1] \to [0,1]$ is chosen continuously so that $f(0) = 1$ and $|\sigma_2(t)| = V_0$. Note that because $\sup_{t \in [0,1]} \mathcal{F}(\sigma_A(t), A) \leq \sup_{t \in [0,1]} \mathbf{F}(\sigma_A(t), A) \leq \eps$, we have that $f$ is well defined and $0 < f(1) < 1$. See Figure \ref{fig:stepone}. 
\begin{figure}[h!]
\centering
\includegraphics[scale=0.7]{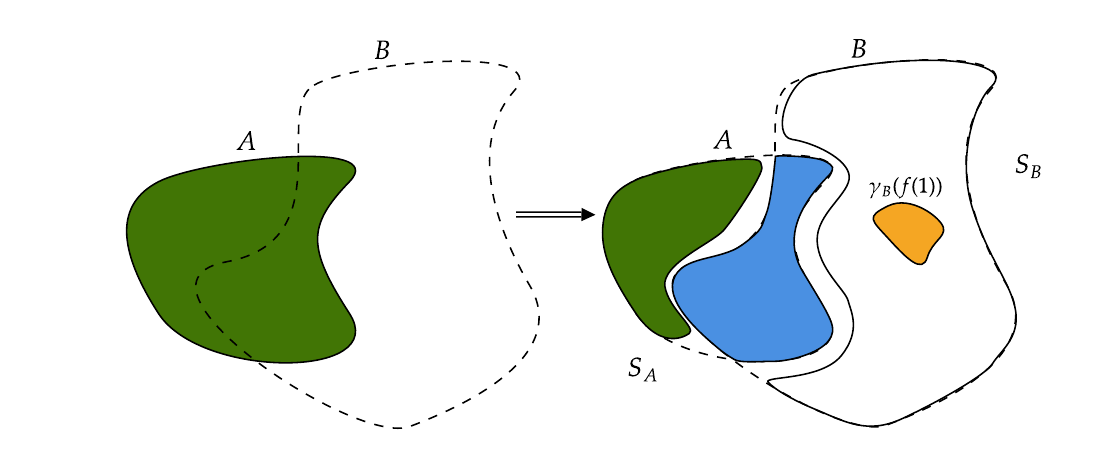}
\caption{$\sigma_1(t)$ visualized}
\label{fig:stepone}
\end{figure}
\nl We now fill in $S_B$ while removing the rest of $S_A$. Consider 
\begin{align*}
\sigma_2&: [0,1] \to \CC(M, \mathbf{F}) \\
\sigma_2(t) &= \begin{cases}
    \sigma_1(2t) & 0 \leq t \leq 1/2 \\
    \sigma_A(1) \backslash \gamma_A(g(t)) \cup \gamma_B(\tilde{f}(t)) & 1/2 \leq t \leq 1 
\end{cases}
\end{align*}
here $g(t)$ is chosen continuously so that $g(1/2) = 1$ and $g$ is monotone non-increasing.
Moreover, $\tilde{f}(t)$ is chosen continuously so that $\tilde{f}(1/2) = f(1)$, $\tilde{f}(1) = 0$, and $\tilde{f}(t)$ is monotone non-decreasing. We can choose both maps so that $|\sigma_2(t)| = V_0$ for all $t$. Note that because $|S_B| < |A \backslash B| = |B \backslash A|$, we must have $g(1) > 0$. See Figure \ref{fig:steptwo}. \nl 
\begin{figure}[h!]
\centering
\includegraphics[scale=0.7]{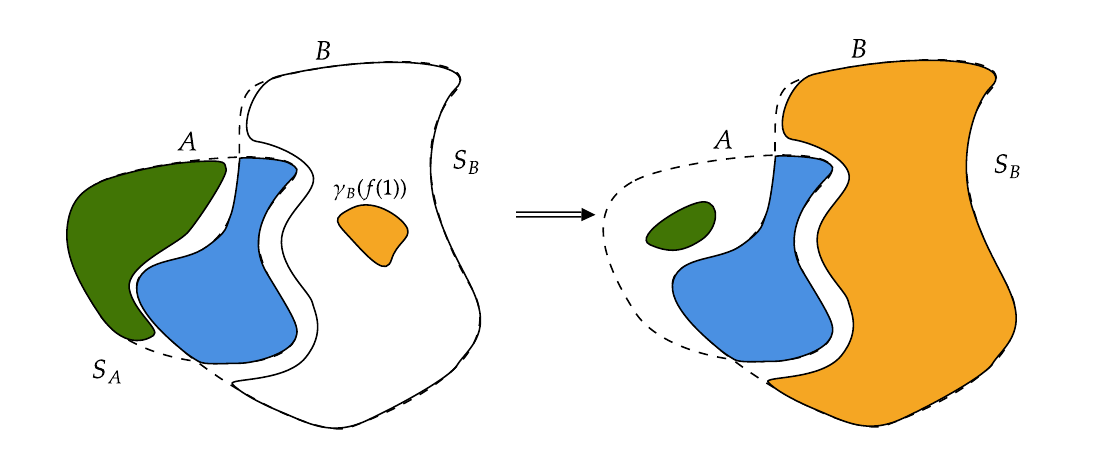}
\caption{$\sigma_2(t)$ visualized}
\label{fig:steptwo}
\end{figure}
\indent Now we finish the map by filling in $B \backslash S_B$ and removing the remainder of $S_A$, i.e. define 
\begin{align*}
\sigma_3&: [0,1] \to \CC(M, \mathbf{F}) \\
\sigma_3(t) &= \begin{cases}
    \sigma_2(2t) & 0 \leq t \leq 1/2 \\
    \sigma_A(1) \backslash \gamma_A(\tilde{g}(t)) \cup \sigma_B(2(1-t)) & 1/2 \leq t \leq 1 
\end{cases}
\end{align*}
Here, $\tilde{g}(1/2) = g(1)$, $\tilde{g}$ is non-increasing, $|\sigma_3(t)| = V_0$ for all $t$ by choice of $\tilde{g}$, and necessarily $\tilde{g}(1) = 0$. See Figure \ref{fig:stepthree}.  This completes the proof.
\begin{figure}[h!]
\centering
\includegraphics[scale=0.7]{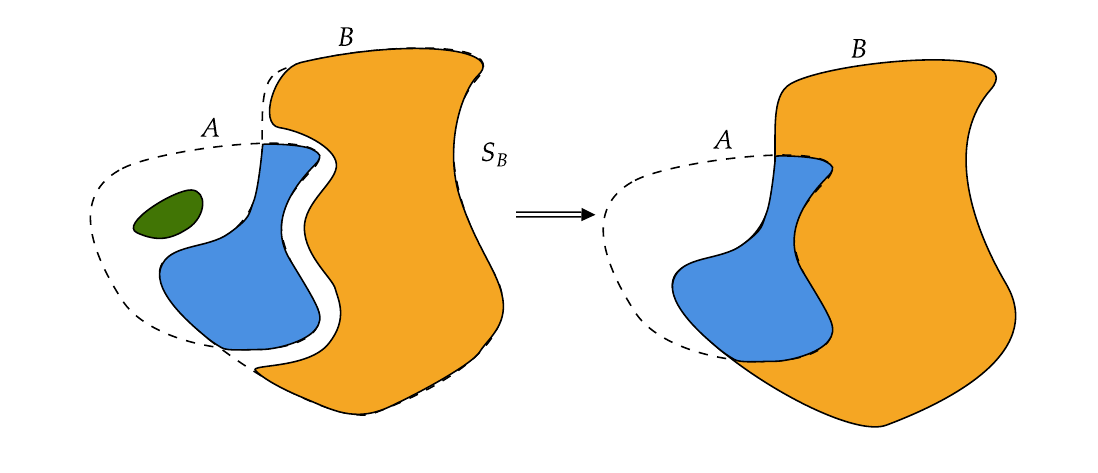}
\caption{$\sigma_3(t)$ visualized}
\label{fig:stepthree}
\end{figure}

\section{Volume Constrained Min-Max for Perimeter}

\subsection{$F^{h,f}$ and $E^f$ functionals}
We recall the $F^{h,f}$ functional introduced by Mazurwoski--Zhou \cite[\S 2, 3]{mazurowski2024infinitely}. Fix a closed, Riemannian manifold $(M, g)$, and $h_0 \in (0, \text{Vol}(M))$. Fix a smooth, non-constant function $f\colon [0,\vol(M)]\to \R$. Also fix a smooth Morse function $h: M\to \R$.  

\begin{definition}[Definition 2.9, \cite{mazurowski2024infinitely}] Given a regular point $x\in M$ for $h$, let $\Gamma(x)$ be the level set of $h$ passing through $x$. Then define $v(h,x)$ to be the vanishing order at $x$ of the mean curvature $H_{\Gamma(x)}$, regarded as a function on $\Gamma(x)$.   
\end{definition} 

\begin{definition}[Definition 2.10, \cite{mazurowski2024infinitely}]
\label{property (T)} Let $h\colon M\to \R$ be a smooth Morse function. We say that $h$ satisfies property (T) provided:
\begin{itemize}
\item[(T)] For every regular point $x$ of $h$, we have $v(h,x) < \infty$. 
\end{itemize}
\end{definition}
In the following, fix a smooth Morse function $h\colon M\to \R$ satisfying property (T). Recall by \cite[Appendix]{mazurowski2024infinitely}, that the set of smooth Morse functions satisfying property (T) is dense in $C^\infty(M)$. 
\begin{definition}[Definition 2.11, \cite{mazurowski2024infinitely}] Define the functional $A^h\colon \CC(M) \to \R$ by 
\[
A^h(\Omega): = \text{Per}(\Omega) - \int_\Omega h. 
\]
Then define the functional $F^{h,f}: \CC(M) \to \R$ by 
\begin{equation} \label{FVolPenFunctional}
F^{h,f}(\Omega) := A^h(\Omega) + f(|\Omega|). 
\end{equation}
In the following, we will sometimes just write $F$ instead of $F^{h,f}$ if the choice of $h$ and $f$ is clear. 

\end{definition}

Assume $\Sigma = \p \Omega$ is a smooth hypersurface in $M$.  Let $\nu$ denote the inward pointing normal vector to $\Sigma$, and let $H$ be the mean curvature of $\Sigma$ with respect to $\nu$.  Let $X$ be a $C^1$ vector field on $M$ and let $\phi_t$ be the associated flow, then
\[
\delta F|_\Omega(X) = \frac{d}{dt}\Big|_{t=0} F(\phi_t(\Omega)) = \int_\Sigma \big(h-H - f'(\vol(\Omega)\big) \langle X, \nu \rangle. 
\]
We also have the following regularity control:
\begin{proposition}[Proposition 2.13, \cite{mazurowski2024infinitely}]\label{prop:touching set}
Assume that $h$ satisfies property $\operatorname{(T)}$. Let $\Sigma$ be an almost-embedded hypersurface with mean curvature $H = h + h_0$ for some constant $h_0$. Then the touching set $S(\Sigma)$ is contained in a countable union of $(n-1)$-dimensional manifolds. 
\end{proposition}
\noindent The $F$ functional extends to a functional on $\vc(M)$ by 
\[
F(V,\Omega) = \|V\|(M) - \int_\Omega h + f(\vol(\Omega)).
\]
If $\Omega_i\in \CC(M)$ is a sequence with $\Omega_i\to \Omega$ and $\vert \p \Omega_i\vert \to V$ then $F(\Omega_i) \to F(V,\Omega)$.  The first variation of $F$ on $\vc(M)$ is given by
\begin{align*}
\delta F|_{(V,\Omega)}(X) &= \frac{d}{dt}\Big|_{t=0} F((\phi_t)_\sharp (V,\Omega)) \\ &= \delta V(X) - \int_\Omega \div(hX) + f'(\vol(\Omega)) \int_{\Omega}\div(X). 
\end{align*}
where $X$ is a $C^1$ vector field on $M$.  We also define the $E^f$ functional as follows
\begin{align} \label{EVolPenFunctional}
E^f& : \vc(M) \to \R \\ \nonumber
E^f(V, \Omega) &:= F^{f,0}(V, \Omega)
\end{align}
though $h = 0$ is not Morse, the above still makes sense. Note that since this paper considers regions enclosing fixed volume, $h_0 \neq \frac{1}{2} \text{Vol}(M)$, we only need to define $E^f$ on the $\vc(M)$ space, as opposed to the $\vz(M, \Z_2)$ space as done in \cite[\S 2.1]{mazurowski2024infinitely}. 
\subsection{Relative Homotopy Classes} \label{relativeHom}
The following is a summarization of \cite[\S 2.3]{mazurowski2024infinitely}. Let $Z \subset X$ be a cubical subcomplex. Given a continuous map $\Phi_0: X\to (\CC(M),\mathbf F)$, we let $\Pi$ be the collection of all sequences of continuous maps $\{\Phi_i: X\to (\CC(M),\mathbf F)\}$ such that, for each $i$, there exists a flat continuous homotopy map 
\begin{gather*}
H_i\colon X\times [0,1] \to (\CC(M),\mathcal F),\\
H_i(x,0) = \Phi_0(x), \\
H_i(x,1) = \Phi_i(x),
\end{gather*}
and, moreover,
\[
\limsup_{i\to \infty} \left[\sup_{(z,t)\in Z\times [0,1]} \mathbf F(H_i(z,t), \Phi_0(z))\right] \to 0
\]
as $i\to \infty$. 

\begin{definition}
Such a sequence $\{\Phi_i\}_{i\in\N}$ is called an $(X,Z)$-homotopy sequence of mappings into $\CC(M)$, and $\Pi$ is called the $(X,Z)$-homotopy class of $\Phi_0$.
\end{definition}

\begin{definition}
Fix a functional $F: \CC(M)\to \R$. Define the min-max value of $\Pi$ with respect to $F$ by
\[
L^{F}(\Pi) = \inf_{\{\Phi_i\}\in \Pi} \limsup_{i\to\infty}\left[\sup_{x\in X} F(\Phi_i(x))\right]. 
\]
A sequence $\{\Phi_i\}\in \Pi$ is called a critical sequence for $F$ if 
\[
L^{F}(\{\Phi_i\}) := \limsup_{i\to \infty} \left(\sup_{x\in X} F(\Phi_i(x))\right) = L^{F}(\Pi). 
\]
The critical set  $\mathcal K(\{\Phi_i\})$ associated to a critical sequence $\{\Phi_i\}$ is the set of all $(V,\Omega)\in \vc(M)$ such that there exist $x_{i_j}\in X$ with  $\vert \Phi_{i_j}(x_{i_j})\vert \to V$, $\Phi_{i_j}(x_{i_j})\to \Omega$, and $F(\Phi_{i_j}(x_{i_j})) \to L^{F}(\Pi)$.  
\end{definition}
\noindent Note that we may also consider relative homotopy classes of maps $\Phi_0: X \to (\CC_{h_0}(M), \mathbf{F})$, in which case the above definitions remain that same after replacing $\CC(M)$ with $\CC_{h_0}(M)$ everywhere. We may emphasize the volume constraint by writing $\Pi_{h_0}$ instead of $\Pi$ in this case.\nl 
\indent We will consider relative homotopy classes for both the $F^{f,h}$ and $E^f$ functionals. We recall the following min-max theorem from Mazurowski--Zhou \cite{mazurowski2024infinitely}:
\begin{theorem}[Thm 1.10,\cite{mazurowski2024infinitely}]
\label{F-min-max} 
Let $(M^{n+1},g)$ be a closed Riemannian manifold of dimension $3\le n+1\le 7$. Let $f\colon [0,\vol(M)]\to\R$ be an arbitrary smooth function, and let $h\colon M\to \R$ be a smooth Morse function satisfying property $\operatorname{(T)}$; see Definition \ref{property (T)}. 
Define $F\colon \CC(M)\to \R$ by 
\[
F(\Omega) = \mathbf{M}(\p \Omega) - \int_\Omega h + f(\vol(\Omega)).
\]
Let $\Pi$ be the $(X,Z)$-homotopy class of a map $\Phi_0\colon X\to (\mathcal C(M),\mathbf F)$. Assume that 
\[
L^F(\Pi) > \sup_{z\in Z} F(\Phi_0(z)).
\]
Then there exists a smooth, almost-embedded hypersurface $\Sigma = \p \Omega$ satisfying 
\[
L^F(\Pi) = F(\Omega).
\]
The mean curvature of $\Sigma$ is given by $H = h - f'(\vol(\Omega))$, and the touching set of $\Sigma$ is contained in a countable union of $(n-1)$-dimensional manifolds. 
\end{theorem} 
We will apply theorem \ref{F-min-max} in the setting of $(X, Z) = ([0,1], \{0,1\})$.
\subsection{Proof of Theorem \ref{MPassThm}}
To begin the proof of Theorem \ref{MPassThm}, we recall some prior work done on strictly stable c-CMC surfaces which bound a Caccioppoli set with fixed volume. For $\Omega \in \CC(M)$ with $\partial \Omega = \Sigma$ a smooth, closed hypersurface, we define 
\[
J_{\Sigma}(\phi) = (\Delta_{\Sigma} + [|A_{\Sigma}|^2 + \Ric_g(\nu, \nu)] )(\phi)
\]
to be the Jacobi operator. Along a family of closed hypersurfaces $\{\Sigma_t\}$ such that $\Sigma_0 = \Sigma$, we have
\[
\frac{d^2}{dt^2} A(\Sigma_t) \Big|_{t = 0} = \int_{\Sigma} - \phi \cdot J_{\Sigma}(\phi) = \int_{\Sigma} |\n \phi|^2 - [|A_{\Sigma}|^2 + \Ric_g(\nu, \nu)]\phi^2 = Q(\phi, \phi)
\]
where $\phi: \Sigma \to \R$ gives the infinitesimal variation of $\Sigma$ at $t = 0$. We are interested in the class of variations which preserve volume, i.e. $\int_{\Sigma} \phi = 0$. Denote this subset of $C^{\infty}(\Sigma)$ by $T_0(\Sigma)$.
\begin{definition}
$\Sigma$ is a strictly stable constant mean curvature surface, if $H_{\Sigma} \equiv c \in \R$, and $\Sigma$ is strictly stable with respect to volume preserving deformations, i.e. 
\[
\phi \in T_0(\Sigma) \backslash \{0\}, \qquad Q(\phi, \phi) > 0
\]
\end{definition}
\noindent By considering the eigenfunctions of $J_{\Sigma}$ on the orthogonal complement of the function $f = 1$, strict stability with respect to volume preserving deformations is equivalent to 
\[
Q(\phi, \phi) \geq C_{\Sigma} \phi^2
\]
for some $C_{\Sigma} > 0$ and $\phi \in T_0(\Sigma)$. We now recall the \textit{quantitative strict stability} proved by Chodosh--Engelstein--Spolaor \cite{chodosh2022riemannian}
\begin{lemma}[Thm 1.4 \cite{chodosh2022riemannian}] \label{QuantIso}
Let $(M^n, g)$ a smooth closed manifold and $\Omega$ a Caccioppoli set with $\text{Vol}(\Omega) = V_0$. Suppose that $\Sigma = \p \Omega$ is a strictly stable c-CMC surface. There exists constants $C, \delta > 0$ (depending only on $M, g, \Sigma$) such that for all other $E \in \CC(M)$ with $|\Omega \Delta E| \leq \delta$ and $\text{Vol}(E) = V_0$, then
\[
\text{Per}(E) - \text{Per}(\Omega) \geq C_0 |E \Delta \Omega|^2.
\]
\end{lemma}
%
%
\noindent The above theorem is a subset of the fully stated \cite[Thm 1.4]{chodosh2022riemannian}. While the original theorem is stated for $g$ analytic, the reader may verify that the subcase of when $\p \Omega = \Sigma$ is strictly stable only requires $g \in C^3$ (see \cite[Lemma 3.3, (ii)]{chodosh2022riemannian}).  We also remark similar work done by \cite{morgan2010stable} Morgan--Ros and \cite{inauen2018quantitative} Inauen--Marchese. We note that the Baire symmetric differenece is the same as the flat norm on Cacciopoli sets as noted in \S \ref{GMTBackground}.
\subsubsection{Existence of a Mountainpass}
As in the premise of Theorem \ref{MPassThm}, suppose $\Omega_1, \Omega_2 \in \CC_{V_0}^{\infty}(M)$ distinct with $V_0 \in (0, \text{Vol}(M))$. Further suppose $\mathbf{M}(\p \Omega_1) \geq \mathbf{M}(\p \Omega_2)$ and $\p \Omega_1$ is a strictly stable, c-CMC hypersurface with respect to volume preserving deformations (note that $\Omega_1$ is automatically smooth in these dimensions by nature of being a strictly stable CMC hypersurface). By Theorem \ref{pathConnectedThm}, there exists a path $\sigma: [0,1] \to \CC_{V_0}(M)$ with $\sigma(0) = \Omega_1, \sigma(1) = \Omega_2$. Let $\Pi_{V_0} = \Pi(\sigma)$, be the corresponding homotopy class relative to $Z = \{0,1\} \subseteq X = [0,1]$. \nl  
\indent Let $h$ be a Morse function, $h$, satisfying property $\operatorname{(T)}$. We prove the below proposition
\begin{proposition} \label{fixedVolMPIneq}
There exist $\tau, \eps_0 > 0$ (both independent of $k$), so that for all $0 \leq \eps < \eps_0$, we have 
\[
L_{\eps}(\Pi_{V_0}) > \max (F_{\eps h}(\Omega_1), F_{\eps h}(\Omega_2)) + \tau
\]
\end{proposition}
\begin{proof}
We have that 
\begin{align*}
F_{\eps h}(\Omega) &= \Per(\Omega) - \eps \int_{\Omega} h  \\
|F_{\eps h}(\Omega_1)| & \in [\Per(\Omega_1) - \eps ||h||_{L^1(M)}, \Per(\Omega_1) + \eps ||h||_{L^1(M)}] \\
|F_{\eps h}(\Omega_2)| & \in [\Per(\Omega_2) - \eps ||h||_{L^1(M)}, \Per(\Omega_2) + \eps ||h||_{L^1(M)}]  \\
\implies \max (F_{\eps h}(\Omega_1), F_{\eps h}(\Omega_2)) & \leq \Per(\Omega_1) + \eps ||h||_{L^1(M)}
\end{align*}
Consider a path in $\CC_{V_0}(M)$, $\gamma: \Omega_1 \to \Omega_2$, which we know exists by Theorem \ref{pathConnectedThm}. Recall the values $C_0$, $\delta$ from Lemma \ref{QuantIso} and choose $\eta = \min\left(\mathcal{F}(\Omega_1, \Omega_2) / 2, \delta \right)$. Let $t$ be the first time such that 
\[
\mathcal{F}(\gamma(t) , \Omega_1) = \eta
\]
which must exist as $\gamma$ is a continuous path (w.r.t. $\mathbf{F}$ and hence w.r.t. $\mathcal{F}$ as well) from $\Omega_1$ to $\Omega_2$. Then
\begin{align*}
F_{\eps h}(\gamma(t)) & = \Per(\gamma(t)) + \eps \int_{\gamma(t)} h \\
F_{\eps h}(\gamma(t)) &\geq \Per(\gamma(t)) - \eps ||h||_{L^1 } \\
& \geq \Per(\Omega_1) - \eps ||h||_{L^1} + C_0 \eta^2 \\
& \geq \max (F_{\eps h}(\Omega_1), F_{\eps h}(\Omega_2)) + [C_0 \eta^2 - 3 \eps ||h||_{L^1}]
\end{align*}
Choosing $\eps$ sufficiently small and setting $\tau = C_0 \eta^2/2$ completes the proof.
\end{proof}
\noindent Let $f_k(x) = k (x - V_0)^2$, and consider $E_k(\Omega) = \Per(\Omega) + f_k(|\Omega|)$ as in equation \eqref{EVolPenFunctional}. As in Mazurowski--Zhou \cite[\S 6.2]{mazurowski2024infinitely}, we now show that the corresponding $L^k$-values converge, when we allow for paths into all Cacciopoli sets. Let $\Pi = \Pi(\sigma)$ be the homotopy class of paths into $\CC(M)$ with respect to $Z = \{0,1\} \subseteq X = [0,1]$. Let $L^{k}(\Pi)$ be the corresponding min-max value for $E^k$, and let $L_0(\Pi_{V_0})$ denote the same min-max value in proposition \ref{fixedVolMPIneq} for $\eps = 0$.
\begin{proposition} \label{limitLVals}
For $\Pi$ the homotopy class defined above
\[
\lim_{k \to \infty} L^k(\Pi) = L_0(\Pi_{V_0})
\]
\end{proposition}
\noindent \begin{rmk}
As a consequence of this proposition, we see that 
\[
L^k(\Pi) > \max (E^k(\Omega_1), E^k(\Omega_2)) + \tau_0
\]
for some $\tau_0 > 0$ independent of $k$ for all $k$ sufficiently large.
\end{rmk}
\begin{proof}[Proof of proposition \ref{limitLVals}]
By considering the fixed volume paths from $\Omega_1 \to \Omega_2$ (again, which exist by Theorem \ref{pathConnectedThm}), we clearly have 
\[
L^k(\Pi) \leq L_0(\Pi_{V_0})
\]
for all $k$, so suppose for contradiction that 
\[
\liminf_{k \to \infty} L^k(\Pi) < L_0(\Pi_{V_0}).
\]
After passing to a subsequence, we can find maps $\Psi_k: [0,1] \to \CC(M)$ 
homotopic to $\sigma$, so that 
\[
\sup_{t \in [0,1]} E_k(\Psi_k(t)) \leq L_0(\Pi_{V_0}) - \eta
\]
for all $k$ sufficiently large and some $\eta > 0$ fixed. In particular, this tells us that 
\[
\sup_{t \in [0,1]} \Per(\Psi_k(t)) \leq L_0(\Pi_{V_0}) - \eta, \qquad \sup_{t \in [0,1]} \Big| |\Psi_k(t)| - V_0 \Big| \leq \sqrt{\frac{L_0(\Pi_{V_0})}{k}}
\]
We have the following straightforward adaptation of Mazurowski--Zhou \cite[Lemma 6.6]{mazurowski2024infinitely} (see also Proposition 5 and 6 in \cite{mazurowski2025half}) to the non half-volume setting:
\begin{lemma}[Lemma 6.6 \cite{mazurowski2024infinitely}] \label{deformLem}
There exists a deformation retract 
\[
\theta: \CC(M) \times [0,1] \to \CC_{V_0}(M)
\]
which is continuous in the $\mathcal{F}$-topology. Moreover, there is a continuous function $w: [0, \infty) \to [0, \infty)$ such that $w(0) = 0$ and 
\[
\Per(\theta(\Omega, 1)) \leq \Per(\Omega) + w(| |\Omega| - V_0)
\]
for any $\Omega \in \CC(M)$.
\end{lemma}
%
%
%
\noindent Let $\Xi_k(t) := \theta(\Psi_k(t), 1)$ which is necessarily homotopic in the flat topology to $\Psi_k$ and hence homotopic to $\sigma$ by the properties of $\theta$. Moreover, by definition of $\theta$, $\Xi_k(0) = \Omega_1$, $\Xi_k(1) = \Omega_2$. For each $k$ and any $\eps > 0$, we can apply discretization and interpolation (see e.g. \cite[Thms 1.11, 1.12]{zhou2020multiplicity}) to produce $\mathbf{F}$-continuous maps, $\{\tilde{\Xi}_k\}$, such that 
\begin{align} \nonumber
\sup_{t \in [0,1]} \mathcal{F}(\Xi_k(t), \tilde{\Xi}_k(t)) &< \eps \\ \nonumber
\sup_{t \in [0,1]} \text{Per}(\tilde{\Xi}_k(t)) & < \sup_{t \in [0,1]} \text{Per}(\Xi_k(t)) + \eps \\ \label{boundaryCondition}
\tilde{\Xi}_k(0) = \Omega_1, &\; \tilde{\Xi}_k(1) = \Omega_2
\end{align}
Note that property \ref{boundaryCondition} is not automatically guaranteed by Theorems 1.11 and 1.12 of Zhou \cite{zhou2020multiplicity}. However, one can check that the discretization process of \cite[Thm 1.11]{zhou2020multiplicity} can be modified so that the corresponding discretized maps and homotopies, $\{\phi_i, \psi_i\}$, satisfy $\phi_i(0) = \Omega_1 = \psi_i(0,s)$ for all $s \in [0,1]$ and $\phi_i(1) = \Omega_2 = \psi_i(1,s)$ for all $s \in [0,1]$ - this actually follows by redefining $\phi_i(0), \phi_i(1), \psi_i(0, s), \psi_i(1,s)$ this way and replacing $\delta_i \to 2 \delta_i$. Then the interpolation process of \cite[Thm 1.12]{zhou2020multiplicity} guarantees property \ref{boundaryCondition} and hence $\tilde{\Xi}_k \in \Pi(\sigma)$. \nl 
\indent Now, define 
\[
c_k = \sup_{\alpha \in [0, \sqrt{L_0(\Pi_{V_0})/k}]} w(\alpha)
\]
and note that by continuity of $w$, $\lim_{k \to \infty} c_k = 0$. Then by Lemma \ref{deformLem}, for $k$ sufficiently large, we have that 
\[
\sup_{t \in [0,1]} \Per(\tilde{\Xi}_k(t)) \leq \sup_{t \in [0,1]} \Per(\Xi_k(t)) + \eps \leq L_0(\Pi_{V_0}) - \eta + c_k + \eps \leq L_0(\Pi_{V_0}) - \eta/2
\]
Moreover
\begin{align*}
\sup_{t \in [0,1]} \mathcal{F}(\Xi_k(t), \tilde{\Xi}_k(t)) &< \eps \\
\implies \sup_{t \in [0,1]} f_k(\tilde{\Xi}_k(t)) & < k \eps^2
\end{align*}
because $|\Xi_k(t)| = V_0$ for all $t$. Choosing $\eps$ sufficiently small (which can be done independent of $k$), we conclude that 
\[
\sup_{t \in [0,1]} E_k(\tilde{\Xi}_k(t)) < L_0(\Pi_{V_0}) - \eta/4
\]
This contradicts the definition of $L_0(\Pi_{V_0})$.
\end{proof}
%
\noindent To complete the proof of Theorem \ref{MPassThm}, we consider the volume penalized functionals 
\begin{align*}
F_{k,\eps}&: \CC(M) \to \R \\
F_{k,\eps}(\Omega) &= F^{\eps h, f_k}(\Omega) = \text{Per}(\Omega) - \eps \int_{\Omega} h + k (|\Omega| - |\Omega_1|)^2
\end{align*}
Now let $L^{k,\eps}(\Pi)$ denote the $L$ value of $\Pi(\sigma)$ corresponding to $F_{k,\eps}$ (on the space of Cacciopoli sets with no volume constraint). Applying propositions \ref{limitLVals} and \ref{fixedVolMPIneq}, we have
\begin{align*}
L^{k,\eps}(\Pi) & > L^{k}(\Pi) - \eps ||h||_{L^1(M)} \\
& > L_0(\Pi_0) - \delta - \eps ||h||_{L^1(M)} \\
& \geq L_{\eps}(\Pi_{V_0)} - \delta - 2 \eps ||h||_{L^1(M)} \\
& \geq \max (F_{\eps h}(\Omega_1), F_{\eps h}(\Omega_2)) + \tau/2 \\
& \geq \max (F_{\eps h, f_k}(\Omega_1), F_{\eps h, f_k}(\Omega_2)) + \tau/2
\end{align*}
where $\delta > 0$ was made arbitrarily small by taking $k$ large. 
Applying Theorem \ref{F-min-max}, we conclude the existence of a $\Omega_{\eps, k} \in \CC(M)$ such that $F_{k,\eps}'(\Omega_k) = 0$ and $\Sigma_{k,\eps} = \partial \Omega_{k,\eps}$ is a smooth, almost embedded hypersurface with mean curvature
\[
H_{\Sigma_{k,\eps}} = \eps h \Big|_{\Sigma_{k,\eps}} - 2k(|\Omega_{\eps,k}| - |\Omega_1|)
\]
We now want to take $\eps \to 0$ and $k \to \infty$ to produce a volume constrained critical point of the perimeter functional. However, the argument is exactly as in Mazurowski--Zhou \S 6.5, with our specific choice of homotopy class and corresponding value $L_0(\Pi_0)$ (replacing $\tilde{\omega}_p$ in their work). We remark that one needs a version of \cite[Prop B.1]{mazurowski2024infinitely} for volumes not equal to $\frac{1}{2} \text{Vol}(M)$, however the exact same proof works for any $V_0 \in (0, \text{Vol}(M))$.\qed 

%
%

\bibliography{main}{}
\bibliographystyle{amsalpha}
\end{document}